\documentclass[a4paper,11pt]{amsart}
\usepackage{amsthm,amsfonts,amsmath,amssymb,enumerate,graphicx}
\usepackage[all]{xy}
\usepackage{multicol}
\usepackage{color}
\usepackage[utf8]{inputenc}
\usepackage[backref]{hyperref}

\newcommand{\C}{{\mathbb C}}

\newcommand{\R}{{\mathbb R}}

\newcommand{\N}{{\mathbb N}}
\newcommand{\K}{{\mathbb K}}

\newcommand{\n}{{\mathfrak n}}

\newcommand{\g}{{\mathfrak g}}
\newcommand{\h}{{\mathfrak h}}

\newcommand\calN{{\mathcal{N}}}
\newcommand\calO{{\mathcal{O}}}

\newcommand\calL{{\mathcal{L}}}

\newtheorem{theorem}{Theorem}[section]

\newtheorem{proposition}[theorem]{Proposition}

\theoremstyle{definition}

\theoremstyle{remark}
\newtheorem{remark}[theorem]{Remark}

\numberwithin{equation}{section}

\begin{document}

\title[Nash-Moser Theorem and rigidity of nilpotent  Lie algebras]{The Nash-Moser Theorem of Hamilton \\
and rigidity of finite dimensional \\ nilpotent Lie algebras }
\author{Alfredo Brega}
\address{CIEM-CONICET, FAMAF-Universidad Nacional de C\'ordoba}
\email{brega@famaf.unc.edu.ar}
\author{Leandro Cagliero}
\address{CIEM-CONICET, FAMAF-Universidad Nacional de C\'ordoba}
\email{cagliero@famaf.unc.edu.ar}
\author{Augusto Chaves Ochoa}
\address{CIEM-CONICET, FAMAF-Universidad Nacional de C\'ordoba}
\email{aeco03@yahoo.com}
\thanks{Partially supported by SECyT-UNC, FONCyT and CONICET grants}

\begin{abstract}
We apply the Nash-Moser theorem for exact sequences of R. Hamilton
to the context of deformations of  Lie algebras and we discuss some aspects
of the scope of this theorem in connection with the polynomial ideal associated
to the variety of nilpotent Lie algebras. This allows us to introduce the space
$H_{k-nil}^2(\g,\g)$, and certain subspaces of it, that provide fine information
about the deformations of $\g$ in the variety of $k$-step nilpotent Lie algebras.

Then we focus on degenerations and rigidity in the variety of $k$-step nilpotent
Lie algebras of dimension $n$ with $n\le7$ and, in particular,
we obtain rigid Lie algebras and rigid curves in the variety of
3-step nilpotent Lie algebras of dimension 7. We also recover some known results
and point out a possible error in a published article related to this subject.
\end{abstract}

\maketitle

\section{Introduction}
In this paper we will assume that all Lie algebras and representations are finite dimensional,
and mostly over $\R$. Here we apply a finite dimensional version of the Nash-Moser theorem
for exact sequences of R. Hamilton to the context of deformations in the variety of nilpotent
Lie algebras. Our main results are described below.

\subsection{The Nash-Moser theorem of R. Hamilton}
A very well known general principle of deformation theory says that given
an (algebraic) structure $\mu$, then
\begin{equation}\label{principle}
H^2(\mu,\mu)=0\;\Rightarrow\; \mu\text{ is rigid, but the converse is not true in general.}
\end{equation}
By definition, an algebraic structure $\mu$ on a $\K$-vector space $V$ is rigid if the
$GL(V)$-orbit of $\mu$, $\calO(\mu)$, is a Zariski open set in the algebraic variety
of all such algebraic structures.

\smallskip

Roughly speaking, when $\K=\R$ or $\C$, an algebraic structure $\mu$ is rigid if
every small perturbation of $\mu$ is isomorphic to $\mu$. More precisely,
it is known that  $\calO(\mu)$ is open in the metric topology if and only if it is
open in the Zariski topology (see \cite[Proposition 17.1]{NR3}, see also
\cite[Proposition 2]{GK}).
As a consequence of this, the principle \eqref{principle} follows from a particular
instance of the Nash-Moser theorem for exact sequences of R. Hamilton as we recall below.
This theorem is stated in \cite{H1} in terms of tame Fr\'echet spaces and it is related
to the inverse function theorem of Nash and Moser \cite{Na,Mo}.
Here, we state
a finite dimensional version of the Nash-Moser theorem
of R. Hamilton.

\begin{theorem}\label{mainthm}
Let $U\subset\R^{m}$ and $V\subset\R^{n}$ be open sets and let
\begin{equation}\label{Ex.Seq.Cinf}
 U \xrightarrow{F}V \xrightarrow{G} \R^k
\end{equation}
be a sequence of $C^{\infty}$ functions such that $G\circ F$ is constant, say $G(F(x))=0$ for all $x\in U$.
Fix $a\in U$
and let $b=F(a)\in V$. If the linear sequence
\begin{equation}\label{Ex.Seq.}
\xymatrix {\R^{m} \ar[r]^{dF|_a} &\R^{n} \ar[r]^{dG|_b} & \R^{k}}
\end{equation}
is exact (in the usual algebraic sense),
there is an open neighborhood $W\subset V\subset\R^{n}$ of $b$ and a $C^{\infty}$ map
$H:W\rightarrow U\subset\R^{m}$ such that $F(H(y))=y$, for all $y\in W$ satisfying $G(y)=0$. That is
\begin{equation*}
\xymatrix{
U    \ar[r]^F       &    V  \ar[r]^G  &    \R^k  &   F\circ H = Id\quad\text{in}\quad\{y\in W:G(y)=0\}.\\
                    &  W \ar[lu]^H\ar@{^{(}->}[u]
}
 \end{equation*}
\end{theorem}
In some sense, this theorem says that the exactness of \eqref{Ex.Seq.} implies a ``local splitting'' of  \eqref{Ex.Seq.Cinf}.

\smallskip

We could not find  this statement in the literature. Also, and remarkably to us,
we do not see this theorem frequently cited in articles dealing with algebraic
structures in the context of \eqref{principle}.

Recently, I. Struchiner pointed out to us that a close result appears in Serre's book
(see in \cite[pp. 89-90]{S} the result attributed to Weil). The statement of this
result is a bit weaker than Theorem \ref{mainthm}, however,
as the referee indicated, Serre's proof of his statement in fact proves the statement of Theorem  \ref{mainthm}.

On the other hand, Hamilton's proof of his version of Theorem \ref{mainthm},
in the context of tame Fr\'echet spaces and tame smooth functions,  is considerably involved.
We came across Hamilton's paper because it is cited in the survey \cite{CSS} of
M. Crainic, F. Sch\"atz and I. Struchiner where the authors address,
in a unified way, several well known problems about rigidity
and stability of Lie algebras and morphisms based on the principle \eqref{principle}.
To address these problems, the authors state and prove some
stability results (see Propositions 4.3, 4.4 and 4.5 in \cite{CSS})
that are phrased in terms of Kuranishi models and non-degenerate zeros of equivariant
sections of vector bundles with group actions. Although the statement of Proposition 4.3 in \cite{CSS}
is more involved than that of Theorem \ref{mainthm}, as in Serre's case,
the proof of Proposition 4.3 of \cite{CSS} provides a proof of Theorem \ref{mainthm}.

We think that both, the Nash-Moser theorem of Hamilton for tame Fr\'echet spaces and the results
about Kuranishi models and equivariant sections of \cite{CSS}, are much deeper than
what is needed to address some problems about rigidity and stability in a finite dimensional context.

We hope that the applications of Theorem \ref{mainthm} presented in this paper  will
 make this classic and important result accessible to more people
working on deformations of algebraic structures.

\subsection{Degenerations and rigidity of nilpotent Lie algebras}
Theorem \ref{mainthm} can be applied to the study of the deformations of any algebraic structure
$\mu$ on a finite dimensional $\R$-vector space $V$ and, in particular, when $\mu$ defines
a Lie algebra structure on $V$.

Let  $\calL_n$ (resp. $\calN_{n}$) be the algebraic variety of all Lie algebra
(resp. nilpotent Lie algebra)  structures $\mu$ on an $n$-dimensional vector space $\g$. Let $\{\g^{i}\}_{i\ge0}$
and $\{\g^{(i)}\}_{i\ge0}$ denote, respectively, the descending central series and the derived
series of $\g$ (i.e. $\g^0=\g^{(0)}=\g$, $g^{i}=[\g^{i-1},\g]$ and $\g^{(i)}=[\g^{(i-1)},\g^{(i-1)}]$)
and let $\calN_{n,k}=\{\g\in\calL_n:\g^k=0\}$ and $\mathcal{S}_{n,k}=\{\g\in\calL_n:\g^{(k)}=0\}$ be, respectively, the subvariety of all (at most)
$k$-step nilpotent and solvable Lie algebras.

\vspace{0.1cm}

Theorem \ref{mainthm} applies to the study of the deformations in a given variety.
For instance,  if we want to study the
deformations of $\mu$ in $\calL_n$ we consider the following sequence of $C^{\infty}$
functions,
\begin{equation}\label{eq.FJ}
GL(\g)\xrightarrow{F}\Lambda^2\g^* \otimes\g \xrightarrow{G=J} \Lambda^3\g^*\otimes \g
\end{equation}
where $F$ is the action of $GL(\g)$ on $\mu$ and $J$ is the Jacobi operator, that is
\begin{align*}
F(g) (x, y)    &= g(\mu (g^{-1} x ,\, g^{-1}  y)), \; &&g\in GL(\g);\\
J(\sigma)(x,y,z)& = \displaystyle\sum_{cyclic}\sigma(\sigma(x,y),z),\; &&\sigma\in
\Lambda^2\g^*\otimes\g \quad \text{and} \quad x,y,z\in\g.
\end{align*}

It turns out that the following portion of the Chevalley-Eilenberg complex for
the adjoint cohomology $\g$,
\begin{equation*}
\g^* \otimes\g\xrightarrow{d_{\mu}^1}\Lambda^2\g^* \otimes\g
\xrightarrow{d_{\mu}^2} \Lambda^3\g^*\otimes \g
\end{equation*}
is the sequence
\begin{equation}\label{eq.dFdJ}
T_I (GL(\g)) \xrightarrow{dF|_I}T_{\mu} (\Lambda^2\g^*\otimes \g)
\xrightarrow{dJ|_{\mu}} T_0 ( \Lambda^3\g^*\otimes \g),
\end{equation}
(see \eqref{Ex.Seq.}) corresponding to \eqref{eq.FJ}.
Therefore, if $H^2(\mu,\mu) = 0$, we obtain from Theorem \ref{mainthm},  that
there exist an open neighborhood $W$ of $\mu$ in $\Lambda^2\g^*\otimes\g$
and a $C^{\infty}$ map $H:W\rightarrow GL(\g)$ such that
\begin{equation}
H(\lambda)\cdot\mu = F(H(\lambda))=\lambda,
\end{equation}
for every $\lambda\in W\cap\{J=0\} = W\cap\calL_n$.
Hence $(\g, \mu)$ is rigid in $\calL_n$
(see also Theorem 5.3 of \cite{CSS}).

\vspace{0.1cm}

On the other hand, if we are interested in the deformations of $\mu$ in $\calN_{n,k}$
we can apply Theorem \ref{mainthm} considering the following $C^{\infty}$ functions,
\begin{equation}\label{}
GL(\g)\xrightarrow{\;\;F\;\;}\Lambda^2\g^* \otimes\g \xrightarrow{\quad G=J\oplus N_k\quad } \Lambda^3\g^*\!\otimes\! \g\;\oplus\;(\g^*)^{\otimes (k+1)}\!\otimes\!\g,
\end{equation}
where $F$ is as above and $N_k:\Lambda^2\g^*\otimes\g \rightarrow (\g^*)^{\otimes(k+1)}\otimes\g$
is given by,
\begin{equation}\label{N_k}
N_k(\sigma)(x_1,\dots,x_{k+1})=\sigma(\dots\sigma(\sigma(x_1,x_2),x_3),\dots,x_{k+1})\; \; \text{for}
\; \; k\ge 1.
\end{equation}
This sequence allows us to introduce the cohomology space $H_{k-nil}^2(\mu,\mu)$
(see \eqref{eq.k-nil}) obtaining that $(\g,\mu)$ is rigid in  $\calN_{n,k}$
whenever $H_{k-nil}^2(\mu,\mu)=0$.

\vspace{0.1cm}

Theorem \ref{mainthm} could be applied in a more subtle way.
Note that $J$ and $N_k$ give rise, when written in coordinates,
to polynomials of degree 2 and $k$ respectively. Let $I_{n,k}$ be the ideal generated
by these polynomials.
It turns out that, depending on $n$ and $k$, there might be polynomials of degree
less than $k$ in the radical $\sqrt{I_{n,k}}$ that are not in
$I_{n,k}$. In this case, if $P$ is such a polynomial,
it can be used as (part of) the function $G$ in Theorem \ref{mainthm}
to describe more precisely $\calN_{n,k}$ which, in turn,
might help to recognize rigid $\mu$'s as points with ``zero cohomology''.
For example, this happens for the ideal of $I_{7,6}$
(which defines $\calN_{7}=\calN_{7,6}$).
Indeed, the polynomial identity $[\g^1,\g^3]=0$, of degree $5$, holds for every
nilpotent Lie algebra $\g$ of dimension 7.
This is discussed in \S\ref{sec.radical} and, in \S\ref{sec.curves in dim7},
we use the identity $[\g^1,\g^3]=0$ and Theorem \ref{mainthm}
to recover the result that states that the only three (two over $\C$) curves in
$\calN_{7}$ are rigid curves. As a byproduct we obtain a curve, consisting of
solvable Lie algebras, that is rigid in the variety of
Lie algebras satisfying  $[\g^1,\g^3]=0$.


The paper also includes an analysis of all rigid Lie algebras in $\calN_{n,k}$ for $n=5,6$.
Some of these results provide examples showing that the converse part of the principle \eqref{principle} is false in $\calN_{5,3}$,  $\calN_{6,3}$,  and $\calN_{6,4}$
(in analogy with the famous example of Richardson \cite{R} in $\calL_{18}$).

Finally,  in \S\ref{sec.Nil_7-3}, we discuss degenerations and rigidity in $\calN_{7,3}$.
As far as we know, the results of this subsection are new. We obtain three rigid Lie
algebras and two (one over $\C$) rigid curves in $\calN_{7,3}$. We also present
degenerations for all Lie algebras $\g\in\calN_{7,3}$ with $\dim H_{3-nil}^2(\g,\g)=1$.
In particular,
we provide a non-trivial deformation of a Lie algebra in $\calN_{7,3}$ which is
claimed to be rigid in \cite{GR}.

\section{Rigidity in the variety $k$-step nilpotent Lie algebras}\label{sec.k-nilp}
Recall from the introduction that
\[
 \calN_{n,k}      =\{\g\in\calL_n:\g^k=0\}
\]
is the subvariety of $\calL_{n}$ of all (at most) $k$-step nilpotent Lie algebras.
Now, we will use Theorem \ref{mainthm} to discuss
rigidity in  $\calN_{n,k}$.

Let $\g$ be a vector space over $\R$ of dimension $n$. For $k\ge 1$ consider the
maps
\begin{equation}\label{eq.N_k}
N_k\in \text{Hom}(\Lambda^2\g^*\otimes \g, \,\, (\g^*)^{\otimes (k+1)}\otimes\g),
\end{equation}
defined inductively as follows,
\begin{align*}
N_1(\mu)&(x_1,x_2)=\mu(x_1,x_2),\\
N_{k}(\mu)&(x_1, \dots, x_k,x_{k+1})=
\mu\big(N_{k-1}(\mu)(x_1, \dots, x_k), x_{k+1}\big),
\end{align*}
where $\mu\in\Lambda^2\g^*\otimes \g$.
It is clear that
\begin{equation}\label{eq.N_{n,k}}
\calN_{n,k}=\{\mu\in\Lambda^2\g^*\otimes \g: \, \, J(\mu)=0 \, \, \text{and} \, \, N_k(\mu)=0\}.
\end{equation}
It is not difficult to see that $dN_k|_{\mu}:\Lambda^2\g^*\otimes\g\rightarrow (\g^*)^{\otimes (k+1)}\otimes\g$
is given by,
\begin{align*}
dN_1|_{\mu}(\sigma)&(x_1,x_2)=\sigma(x_1,x_2),\\
dN_2|_{\mu}(\sigma)&(x_1,x_2,x_3)=\mu(\sigma(x_1,x_2),x_3)+\sigma(\mu(x_1,x_2),x_3),\\
dN_3|_{\mu}(\sigma)&(x_1,x_2,x_3,x_4)=\mu(\mu(\sigma(x_1,x_2),x_3),x_4)+\mu(\sigma(\mu(x_1,x_2),x_3),x_4) \\
 &\hspace{6.5cm} + \sigma(\mu(\mu(x_1,x_2),x_3),x_4),
\end{align*}
and so on for $k\ge 4$.

\vspace{0.1cm}

Since $\text{Im}(d_{\mu}^1)\subset\text{Ker}(d_{\mu}^2)\cap\text{Ker}(dN_k|_{\mu})$,
in analogy with the definition of the second cohomology space $H^2(\g,\g)$ we define
\begin{equation}\label{eq.k-nil}
H_{k-nil}^2(\g,\g)=\frac{\text{Ker}(d_{\mu}^2)\cap\text{Ker}(dN_k|_{\mu})}{\text{Im}(d_{\mu}^1)}
\end{equation}
for any  $k$-step nilpotent Lie algebra $(\g, \mu)$. In the following theorem we apply
Theorem \ref{mainthm} to obtain a rigidity result in line with the principle \eqref{principle}.

\begin{theorem}\label{k-rigidity}
Let $(\g, \mu)$ be a $k$-step nilpotent Lie algebra over $\R$.
If $H_{k-nil}^{2}(\g,\g) = 0$, then
there is a neighborhood $W$ of $\mu$ in $\Lambda^2\g^*\otimes\g$ and a smooth map
$H: W \rightarrow GL(\g)$ such that $H(\lambda)\cdot \mu = \lambda$ for every
$\lambda\in W\cap\calN_{n,k}$. Hence, $(\g, \mu)$ is rigid in $\calN_{n,k}$.
\end{theorem}

\begin{proof} Let $(\g, \mu)$ be a $k$-step nilpotent Lie algebra over $\R$ and let $F$ and $J$ be
as in \eqref{eq.FJ}. Consider the following  3-term $C^{\infty}$-chain complex
\begin{equation}\label{eq.FG}
GL(\g)\xrightarrow{\;\;F\;\;}\Lambda^2\g^* \otimes\g \xrightarrow{\quad G=J\oplus N_k\quad }\Big(\Lambda^3\g^*\otimes\g\Big)
\oplus\Big((\g^*)^{\otimes (k+1)}\otimes\g\Big).
\end{equation}
Since $\mu$ is a $k$-step nilpotent Lie algebra we have $G(F(g))=0$ for $g\in GL(\g)$.
Take $a=I\in GL(\g)$ and $b=F(I)=\mu\in\Lambda^2\g^*\otimes\g$ in Theorem \ref{mainthm}.
Since
\begin{equation*}\label{eq.dFdG.1}
dF|_I =d_{\mu}^1 \,\quad \, \text{and} \,\quad \, dG|_{\mu}=d_{\mu}^2 \oplus dN_k|_{\mu},
\end{equation*}
it follows that the sequence
\begin{equation}\label{eq.dFdG.2}
\g^*\otimes \g \xrightarrow{dF|_I}T_{\mu} (\Lambda^2\g^*\otimes \g)
\xrightarrow{dG|_{\mu}} T_0 \Big(\Lambda^3\g^*\otimes \g\Big)\oplus
T_0 \Big((\g^*)^{\otimes (k+1)}\otimes\g\Big),
\end{equation}
is exact if and only if $H_{k-nil}^2(\g,\g) = 0$. Hence, from Theorem \ref{mainthm}, it follows
that if $H_{k-nil}^2(\g,\g) = 0$ there exist a neighborhood $W$ of $\mu$ in $\Lambda^2\g^*\otimes\g$
and a $C^{\infty}$ function $H:W\rightarrow GL(\g)$ such that
\begin{equation}
H(\lambda)\cdot\mu = F(H(\lambda))=\lambda,
\end{equation}
for every $\lambda\in W\cap\{J\oplus N_k=0\} = W\cap\calN_{n,k}$ (see \eqref{eq.N_{n,k}}).
Hence $(\g, \mu)$ is rigid in $\calN_{n,k}$, as we wanted to proof.
\end{proof}

It is natural to ask how to extend  \eqref{eq.dFdG.2} to an unbounded
chain complex. Among other things, this would facilitate the computation of
$H_{k-nil}^2(\g,\g)$ by using homological algebra tools.
We have been able to do this for $2$-step nilpotent Lie algebras and this work
is still in progress.
We point out that M. Vergne \cite{V} introduced, as derived functors,
cohomology spaces  $N^p_{k}(\mu,\mu)$ (for $k$-step nilpotent Lie algebras) that can be viewed
as certain subspaces of the regular cohomology spaces
$H^p(\mu,\mu)$. For $p=2$, this subspace, in principle,
is not quite the same as $H_{k-nil}^2(\mu,\mu)$ as it consists of classes that contain a
representative satisfying certain property (which depends on the representative).
It would be interesting to understand better the connection between $N^2_{k}(\mu,\mu)$ and $H_{k-nil}^2(\mu,\mu)$.

\section{The radical of the polynomial ideal defining  $\calN_{n,k}$}\label{sec.radical}
Let $\g$ be a fixed $n$-dimensional vector space over $\R$. For $\mu\in\Lambda^2\g^*\otimes \g$
we know that the coordinates of $J(\mu)$ and $N_k(\mu)$, when expressed in terms of a basis of $\g$, are, respectively, polynomials of degree 2 and $k$ in the structure constants of $\mu$, (the number of these polynomials and the number of variables depend on $n$). Let $I_{n,k}$
denote the ideal generated by these polynomials. In this subsection we will discuss briefly some algebraic properties of $I_{n,k}$ that depend on $n$. Some of these results are needed in \S\ref{sec.curves in dim7}.

\vspace{0.1cm}

For $k\ge 3$ consider the polynomial in $\mu\in\Lambda^2\g^*\otimes \g$ given by,
\begin{equation}\label{SN_k}
SN_k(\mu)(x_1,\dots,x_{k+1})=\mu\big(\mu(x_1,x_2),N_{k-2}(\mu)(x_3,\dots,x_{k+1})\big),
\end{equation}
and define
\[
\mathcal{SN}_{n,k}=\{\mu\in\Lambda^2\g^*\otimes \g: \, \,  J(\mu)=0 \, \, \, \text{and} \, \, \, SN_{k}(\mu)=0\}
\subset \mathcal{L}_n.
\]
It is clear that if $\mu\in\mathcal{SN}_{n,k}$ then $\mu$ defines a solvable Lie algebra.
More precisely we have,
\begin{equation}\label{eq.SN_{n,k}1}
 \mathcal{N}_{n,k}\subset\mathcal{SN}_{n,k}\subset
 \mathcal{S}_{n,\lceil \log_2(k-1)\rceil+1}\subset\mathcal{S}_{n,k-1}\subset\mathcal{L}_n.
\end{equation}
This follows since $\g^{(i)}\subset\g^{2^i-1}$ for $i\in\N$, and because
$N_{k}(\mu)=J(\mu)=0$ imply $SN_{k}(\mu)=0$.

\vspace{0.1cm}

The main goal of this subsection is to point out that
\begin{align}\label{eq.SN_{n,k}3}
\text{in general }& \mathcal{N}_{n,k+1}\not\subset\mathcal{SN}_{n,k}, \quad \text{ for }  k<n-1,\\\label{eq.SN_{n,k}2}
\text{but  }& \mathcal{N}_{n,k+1}\subset\mathcal{SN}_{n,k}, \quad \text{ for  certain }  k<n-1,
\end{align}
(observe that $\mathcal{N}_{n,k}=\mathcal{N}_{n,n-1}\subset\mathcal{SN}_{n,n-1}$ for $k\ge n$).
The inclusion \eqref{eq.SN_{n,k}2} is remarkable to us since it provides ``unexpected'' polynomials of degree
$k$ (those coming from $SN_k$) that vanish on $\mathcal{N}_{n,k+1}$. Moreover, it reveals some instances where the polynomial
ideal $I_{n,k+1}$ is not radical.

\vspace{0.1cm}

More precisely, we will show next that
\begin{align*}
\mathcal{N}_{5,4}\subset\mathcal{SN}_{5,3} && \mathcal{N}_{6,4}\subset\mathcal{SN}_{6,3} && \mathcal{N}_{7,4}\subset\mathcal{SN}_{7,3} && \mathcal{N}_{8,4}\not\subset\mathcal{SN}_{8,3} \\
                                               && \mathcal{N}_{6,5}\not\subset\mathcal{SN}_{6,4}  && \mathcal{N}_{7,5}\not\subset\mathcal{SN}_{7,4} && \mathcal{N}_{8,5}\not\subset\mathcal{SN}_{8,4}  \\
                                               &&                                        && \mathcal{N}_{7,6}\subset\mathcal{SN}_{7,5}  && \text{we don't know}
\end{align*}
and, in general,
\begin{equation}\label{eq.notin}
 \mathcal{N}_{2(k+1),k+1}\not\subset\mathcal{SN}_{2(k+1),k},\quad\text{for all $k\ge2$}.
\end{equation}
We notice that, once
$\mathcal{N}_{n_0,k+1}\not\subset\mathcal{SN}_{n_0,k}$ for some $n_0$ then
$\mathcal{N}_{n,k+1}\not\subset\mathcal{SN}_{n,k}$
for all $n\ge n_0$.
Therefore \eqref{eq.notin} implies that, given $k$, $\mathcal{N}_{n,k+1}\subset\mathcal{SN}_{n,k}$ occurs only for a finite number of $n$'s.

\subsection{Proof of $\mathcal{N}_{n,5}\not\subset\mathcal{SN}_{n,4}$ for $n\ge6$}
Let $\g$ be the Lie algebra denoted by $12346_E$ in \cite{Se} and by $\g_{6,14}$ in
\cite{CdGS}.
The structure table of $\g$ is the following,
\[
\g:\;[a,b] = c,\; [a,c] = d,\; [a,d] = e,\; [b,c] = e,\; [b,e] = f,\; [c,d] =- f.
\]
It is clear that $\g\in\mathcal{N}_{6,5}$
(in fact $\g$ is rigid in $\mathcal{N}_{6,5}$, see \S\ref{sec.dim6})
but $\g\not\in\mathcal{SN}_{6,4}$, since $\big[[a,b],[[a,b],a]\big]=f\ne0$.
Therefore, $\mathcal{N}_{6,5}\not\subset\mathcal{SN}_{6,4}$ and
this implies that $\mathcal{N}_{n,5}\not\subset\mathcal{SN}_{n,4}$ for all $n\ge6$.

\subsection{Proof of $\mathcal{N}_{2(m+1),m+1}\not\subset\mathcal{SN}_{2(m+1),m}$  for all $m\ge2$}
Let $\mathfrak{h}_m$ be the Heisenberg Lie algebra
$[x_i,y_i]=z$ ($i=1,\dots,m$) and let $D\in\text{Der}(\mathfrak{h}_m)$ be the derivation defined by
\begin{align*}
D(x_i)&=x_{i+1},\quad i=1,\dots,m-1;\\
D(y_i)&=-y_{i-1},\quad i=2,\dots,m.
\end{align*}
Then $\n=\R D\ltimes\h_m$ is an $(m+1)$-step nilpotent Lie algebra of dimension $2m+2$, hence $\n\in\mathcal{N}_{2m+2,m+1}$.
However, it is straightforward to verify that $\n\not\in\mathcal{SN}_{2m+2,m}$.
In particular, this  shows that
\begin{align*}
 \mathcal{N}_{n,4}&\not\subset\mathcal{SN}_{n,3}\quad\text{for $n=8$ and }\\
 \mathcal{N}_{n,6}&\not\subset\mathcal{SN}_{n,5}\quad\text{for $n=12$.}
\end{align*}

\medskip

\subsection{Proof of $\mathcal{N}_{n,4}\subset\mathcal{SN}_{n,3}$ for $n=5,6,7$ and $\mathcal{N}_{7,6}\subset\mathcal{SN}_{7,5}$}
Now we will show that
\begin{equation}\label{eq.NSN}
 \mathcal{N}_{n,4}\subset\mathcal{SN}_{n,3},\;n=5,6,7; \quad\text{ and }\quad
 \mathcal{N}_{7,6}\subset\mathcal{SN}_{7,5}.
\end{equation}

One way to prove \eqref{eq.NSN} is to use the classification of all nilpotent
Lie algebras of dimension $\le 7$ and verify \eqref{eq.NSN} case by case.
We do not know whether there is an elegant argument proving \eqref{eq.NSN}.

Another way to prove \eqref{eq.NSN}, which is more interesting to us in
the context of this paper, is to show that the radical of the
ideals $I_{n,4}$ and $I_{n,6}$ (defining  $\mathcal{N}_{n,4}$ and $\mathcal{N}_{n,6}$)
contain, respectively, the ideals defining $\mathcal{SN}_{n,3}$ and $\mathcal{SN}_{n,5}$,
for the given values of $n$. We discuss this approach next.

\vspace{0.1cm}

If $\g$ is a nilpotent Lie algebra of dimension $n$ with bracket given by $\mu$, then
we know (by Lie's Theorem) that $\g$ admits a basis $\{e_i\}$ such that
\[
 \mu(e_i,e_j)=\sum_{k=j+1}^{n}t_{i,j,k}\;e_k,\quad\text{ for $i<j$}.
\]
It is straightforward to see that the coordinates of $J(\mu)(e_{i_1},e_{i_2},e_{i_3})$ and
either $N_k(\mu)(e_{i_1},\dots,e_{i_{k+1}})$ or
$SN_k(\mu)(e_{i_1},\dots,e_{i_{k+1}})$,
are polynomials of degree 2 and $k$, respectively, in the variables $t_{i,j,k}$.
Now we consider separately the values of $n$ given in \eqref{eq.NSN}.

\smallskip

Case $n=5$. Here $I_{5,4}$ is generated by
\[
 P_1=t_{1, 2, 3}t_{3, 4, 5}\quad\text{and}\quad P_2=t_{1, 2, 4}t_{3, 4, 5} + t_{2, 3, 4}t_{1, 4, 5} - t_{1, 3, 4}t_{2, 4, 5},
\]
(notice that the condition $N_4=0$ is trivial for $n=5$). The condition $SN_3=0$ is given
by $Q_1=Q_2=0$ where
\[
 Q_1=t _{1, 2, 3} t _{1, 3, 4} t _{3, 4, 5}\quad\text{and}\quad Q_2=t _{1, 2, 3} t _{2, 3, 4} t _{3, 4, 5}.
\]
Then, since $P_1$ divides $Q_1$ and $Q_2$ we conclude that $\mathcal{N}_{5,4}\subset\mathcal{SN}_{5,3}$.

\smallskip

Case $n=6$.
The ideal $I_{6,4}$ is generated by the polynomials below. Those having degree 2
correspond to $J$ and the one's of degree 4 correspond to $N_4$.

\smallskip

\noindent
Degree 2:
\begin{align*}
& t_{1, 2, 3}t_{3, 4, 5}, \\
& t_{1, 3, 4}t_{4, 5, 6},  \\
& t_{2, 3, 4}t_{4, 5, 6},  \\
& t_{1, 2, 3}t_{3, 5, 6}+t_{1, 2, 4}t_{4, 5, 6},  \\
& t_{1, 2, 4}t_{3, 4, 5}+t_{2, 3, 4}t_{1, 4, 5}-t_{1, 3, 4}t_{2, 4, 5},  \\
& t_{1, 3, 5}t_{4, 5, 6}+t_{3, 4, 5}t_{1, 5, 6}-t_{1, 4, 5}t_{3, 5, 6},  \\
& t_{2, 3, 5}t_{4, 5, 6}+t_{3, 4, 5}t_{2, 5, 6}-t_{2, 4, 5}t_{3, 5, 6},  \\
& t_{1, 2, 3}t_{3, 4, 6}-t_{1, 2, 5}t_{4, 5, 6}-t_{2, 4, 5}t_{1, 5, 6}+t_{1, 4, 5}t_{2, 5, 6},  \\
& t_{1, 2, 4}t_{3, 4, 6}+t_{2, 3, 4}t_{1, 4, 6}-t_{1, 3, 4}t_{2, 4, 6}+t_{1, 2, 5}t_{3, 5, 6}+t_{2, 3, 5}t_{1, 5, 6}-t_{1, 3, 5}t_{2, 5, 6}
\end{align*}

\noindent
Degree 4:
\begin{align*}
t_{1, 2, 3}t_{1, 3, 4}t_{1, 4, 5}t_{1, 5, 6}, && t_{1, 2, 3}t_{2, 3, 4}t_{1, 4, 5}t_{1, 5, 6}, \\
t_{1, 2, 3}t_{1, 3, 4}t_{1, 4, 5}t_{2, 5, 6}, && t_{1, 2, 3}t_{2, 3, 4}t_{1, 4, 5}t_{2, 5, 6}, \\
t_{1, 2, 3}t_{1, 3, 4}t_{1, 4, 5}t_{3, 5, 6}, && t_{1, 2, 3}t_{2, 3, 4}t_{1, 4, 5}t_{3, 5, 6}, \\
t_{1, 2, 3}t_{1, 3, 4}t_{1, 4, 5}t_{4, 5, 6}, && t_{1, 2, 3}t_{2, 3, 4}t_{1, 4, 5}t_{4, 5, 6}, \\
t_{1, 2, 3}t_{1, 3, 4}t_{2, 4, 5}t_{1, 5, 6}, && t_{1, 2, 3}t_{2, 3, 4}t_{2, 4, 5}t_{1, 5, 6}, \\
t_{1, 2, 3}t_{1, 3, 4}t_{2, 4, 5}t_{2, 5, 6}, && t_{1, 2, 3}t_{2, 3, 4}t_{2, 4, 5}t_{2, 5, 6}, \\
t_{1, 2, 3}t_{1, 3, 4}t_{2, 4, 5}t_{3, 5, 6}, && t_{1, 2, 3}t_{2, 3, 4}t_{2, 4, 5}t_{3, 5, 6}, \\
t_{1, 2, 3}t_{1, 3, 4}t_{2, 4, 5}t_{4, 5, 6}, && t_{1, 2, 3}t_{2, 3, 4}t_{2, 4, 5}t_{4, 5, 6}, \\
t_{1, 2, 3}t_{1, 3, 4}t_{3, 4, 5}t_{1, 5, 6}, && t_{1, 2, 3}t_{2, 3, 4}t_{3, 4, 5}t_{1, 5, 6}, \\
t_{1, 2, 3}t_{1, 3, 4}t_{3, 4, 5}t_{2, 5, 6}, && t_{1, 2, 3}t_{2, 3, 4}t_{3, 4, 5}t_{2, 5, 6}, \\
t_{1, 2, 3}t_{1, 3, 4}t_{3, 4, 5}t_{3, 5, 6}, && t_{1, 2, 3}t_{2, 3, 4}t_{3, 4, 5}t_{3, 5, 6}, \\
t_{1, 2, 3}t_{1, 3, 4}t_{3, 4, 5}t_{4, 5, 6}, && t_{1, 2, 3}t_{2, 3, 4}t_{3, 4, 5}t_{4, 5, 6}.
\end{align*}
On the other hand, the condition $SN_3=0$ is $\{Q_i=0:i=1,\dots,14\}$ where the $Q_i$'s are the following polynomials of degree 3:
\begin{align*}
&Q_1 =t_{1, 2, 3}t_{1, 3, 4}t_{3, 4, 5}, \hspace{4.5cm}Q_2 =t_{1, 2, 3}t_{2, 3, 4}t_{3, 4, 5}, \\
&Q_3 =t_{1, 3, 4}t_{1, 4, 5}t_{4, 5, 6}, \hspace{4.5cm}Q_4 =t_{1, 3, 4}t_{2, 4, 5}t_{4, 5, 6},  \\
&Q_5 =t_{1, 3, 4}t_{3, 4, 5}t_{4, 5, 6}, \hspace{4.5cm}Q_6 =t_{1, 4, 5}t_{2, 3, 4}t_{4, 5, 6},  \\
&Q_7 =t_{2, 3, 4}t_{2, 4, 5}t_{4, 5, 6}, \hspace{4.5cm}Q_8 =t_{2, 3, 4}t_{3, 4, 5}t_{4, 5, 6},  \\
&Q_9 =(t_{1, 3, 4}t_{2, 3, 5} - t_{1, 3, 5}t_{2, 3, 4} ) t_{4, 5, 6},  \\
&Q_{10} =t_{1, 2, 3}t_{1, 4, 5}t_{3, 5, 6} + t_{1, 2, 4}t_{1, 4, 5}t_{4, 5, 6},  \\
&Q_{11} =t_{1, 2, 3}t_{2, 4, 5}t_{3, 5, 6} + t_{1, 2, 4}t_{2, 4, 5}t_{4, 5, 6},  \\
&Q_{12} =t_{1, 2, 3}t_{3, 4, 5}t_{3, 5, 6} + t_{1, 2, 4}t_{3, 4, 5}t_{4, 5, 6},  \\
&Q_{13} =t_{1, 2, 3}t_{1, 3, 4}t_{3, 4, 6} + t_{1, 2, 3}t_{1, 3, 5}t_{3, 5, 6} + t_{1, 2, 4}t_{1, 3, 5}t_{4, 5, 6} - t_{1, 2, 5}t_{1, 3, 4}t_{4, 5, 6}, \\
&Q_{14} =t_{1, 2, 3}t_{2, 3, 4}t_{3, 4, 6} + t_{1, 2, 3}t_{2, 3, 5}t_{3, 5, 6} + t_{1, 2, 4}t_{2, 3, 5}t_{4, 5, 6} - t_{1, 2, 5}t_{2, 3, 4}t_{4, 5, 6}
\end{align*}

\smallskip

\noindent
It turns out that $Q_i\in I_{6,4}$ for $i=1,\dots,12$, that $Q_{13},Q_{14}\not\in I_{6,4}$
but $Q_{13}^2,Q_{14}^2\in I_{6,4}$. This shows that $\mathcal{N}_{6,4}\subset\mathcal{SN}_{6,3}$.

\medskip

What it is surprising to us is that $Q_{13},Q_{14}\not\in I_{6,4}$. This shows that
$I_{6,4}$ is not a radical ideal. To prove this we assume in all the polynomials
that,
\[
 t_{1, 2, 4}= t_{1, 3, 4}= t_{1, 4, 5}= t_{1, 4, 6}=
t_{2, 3, 5}= t_{2, 5, 6}= t_{3, 4, 5}= t_{3, 5, 6}=t_{4, 5, 6}=0.
\]
Under this assumption, $I_{6,4}$ is generated by the polynomials,
\begin{equation}\label{eq.ideal_res}
t_{1, 2, 3}t_{2, 3, 4}t_{2, 4, 5}t_{1, 5, 6} \, \, \quad \text{and} \,\, \quad
t_{1, 2, 3}t_{3, 4, 6}-t_{2, 4, 5}t_{1, 5, 6},
\end{equation}
while  $Q_{14}$ becomes
\[
 t_{1, 2, 3}t_{2, 3, 4}t_{3, 4, 6}.
\]
A straightforward computation shows that $Q_{14}$ does not belong to the ideal generated
by the polynomials in \eqref{eq.ideal_res}. If instead of assuming $t_{1, 3, 4}= 0$ we
assume that $t_{2, 3, 4}=0$ we obtain that $Q_{13}\not\in I_{6,4}$.

\medskip

Case $n=7$. In analogy to the previous cases, we have verified with the software
\texttt{Singular} that:
\begin{enumerate}[-]
 \item
the square of the ideal defining $\mathcal{SN}_{7,3}$ is contained in $I_{7,4}$,
 \item
the square of the ideal defining $\mathcal{SN}_{7,5}$ is contained in $I_{7,6}$,
\item there are polynomials defining $\mathcal{SN}_{7,3}$ not belonging to $I_{7,4}$,
as well as polynomials defining $\mathcal{SN}_{7,5}$ not belonging to $I_{7,6}$.
\end{enumerate}
Again, $I_{7,4}$ and $I_{7,6}$ are not radical ideals for $n=7$.

\vspace{0.1cm}

This shows that, for $n\le7$, $SN_3(\mu)$ produces polynomials of degree 3 in the
radical of the ideal generated by  $J(\mu)$ and $N_4(\mu)$; and  $SN_5(\mu)$ produces
polynomials of degree 5 in the radical of the ideal generated by  $J(\mu)$ and $N_6(\mu)$.
This give rise to the following question:

\medskip

\noindent
\textbf{Questions.} Given $n>k$:
\begin{enumerate}[-]
\item Is the ideal $I_{n,k}$ radical?

\item Are there polynomials of degree less than $k$ in $\sqrt{I_{n,k}}$ which
are not in $I_{n,k}$?

\item Does $SN_{k-1}$ vanishes on $\calN_{n,k}$?
Or more generally, is there any other polynomial identity of degree less that $k$
(different from Jacobi) vanishing on $\calN_{n,k}$?
\end{enumerate}

\noindent
For example, in this section we have shown that
\begin{enumerate}[-]
 \item $SN_3$ vanishes in $\mathcal{N}_{n,4}$ for $n=5,6,7$ but does not for $n=8$;
 \item $SN_5$ vanishes in $\mathcal{N}_{7,6}$ but does not vanishes in $\mathcal{N}_{12,6}$.
 \end{enumerate}

Above questions are motivated by the idea of obtaining a polynomial $P(\mu)$, simpler
than $J(\mu)\oplus N_{k}$, to be used as (part of) the function $G$ in Theorem \ref{mainthm}
(or more precisely, in the display \eqref{eq.FG} of Theorem \ref{k-rigidity}) to study
rigidity in $\calN_{n,k}$.
This could help to recognize more easily rigid Lie algebras in $\mathcal{N}_{n,k}$.
This tool is used in \S\ref{sec.curves in dim7} to find rigid curves in
$\mathcal{N}_{7,k}$, for $k=3,6$.

\section{Deformations and rigidity in $\calN_{n,k}$, for $n=5,6,7$}
In this section we will consider several structure tables of Lie algebras.
In order to shorten the description of these tables we will denote by $ab$ the
Lie bracket $[a,b]$.
\subsection{A rigid Lie algebra with $H_{k-nil}^2(\g,\g)\ne0$.}
It is well known that a Lie algebra $\g$ may be rigid but fail to satisfy $H^2(\g,\g)=0$;
the examples are in general involved, see for instance \cite{R}. In this subsection
we present a Lie algebra $\g$ that is rigid in the variety $\mathcal{N}_{5,3}$  but $H_{3-nil}^2(\g,\g)\ne0$.

\vspace{0.1cm}

There are only eight non-abelian Lie algebras of dimension 5 over the real numbers \cite{dG}.
If $\mathfrak{f}_n$ denotes the standard filiform Lie algebra of dimension $n$,
these Lie algebras are

\begin{center}
\begin{tabular}{llc}
$k$    & $k$-step Lie algebra $\g$ & $\dim H_{k-nil}^2(\g, \g)$ \\
\hline \\[-2mm]
2 & $\mathfrak{f}_{3}\oplus\R^2$              & $11=20-9$ \\
       & $\g_{5,1}:\;a b  = e,\; c d  = e$. & $0=10-10$ \\
       & $\g_{5,2}:\;a b  = d,\; a c  = e$. & $0=12-12$ \\[2mm]
3 & $\mathfrak{f}_{4}\oplus\R$                                       & $4=18-14$ \\
       & $\g_{5,3}:\;a b  = d,\; a d  = e,\; b c  = e$. & $2=17-15$ \\
       & $\g_{5,4}:\;a b  = c,\; a c  = d,\; b c  = e$. & $0=15-15$ \\[2mm]
4 & $\mathfrak{f}_{5}$                                                       & $1=17-16$ \\
       & $\g_{5,6}:\;a b  = c,\; a c  = d,\; a d  = e$,\;  $b c  = e$. & $0=17-17$ \\[3mm]
\end{tabular}

The list of 5-dimensional non-abelian nilpotent Lie algebras
\end{center}

\medskip

\noindent
and the Hasse diagram is (see \cite{GO}),
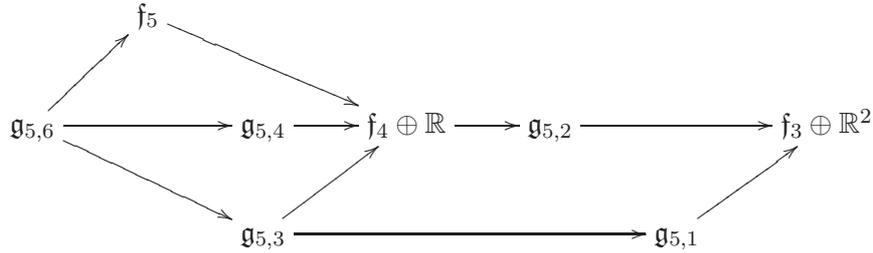
\[
\xymatrix{
                                &\mathfrak{f}_{5}\ar[drr]                 &         \\
\g_{5,6} \ar[rr]\ar[drr]\ar[ur] && \g_{5,4}\ar[r]        &  \mathfrak{f}_{4}\oplus\R \ar[r]  &\g_{5,2} \ar[rr] && \mathfrak{f}_{3}\oplus\R^{2} \\
                                && \g_{5,3}\ar[rrr]\ar[ur] &&&\g_{5,1}  \ar[ur]
                                }
\]
This shows that $\g_{5,3}$ is rigid in the variety of 3-step nilpotent Lie algebras
but it tuns out that $H_{3-nil}^2(\g_{5,3}, \g_{5,3})=\text{span}\{\nu_1,\nu_2\}$ with
\begin{align*}
 \nu_1(b,c)&=c,  \\
  \nu_2(a,b)&=b, \qquad   \nu_2(a,c)=-c , \qquad    \nu_2(a,d)=-d.
\end{align*}
In fact, if $\mu$ is the bracket of $\g_{5,3}$, then $\mu+t\nu_1$ and $\mu+t\nu_2$ are solvable deformations of $\g_{5,3}$.

\subsection{Rigid nilpotent Lie algebras in $\mathcal{N}_{6,k}$}\label{sec.dim6}
The Hasse diagram of the 6-dimensional nilpotent Lie algebras is given in  \cite{Se}.
There are 34 real 6-dimensional nilpotent Lie algebras \cite{CdGS} and being this a finite number it follows,
as in dimension 5, that a Lie algebra is rigid in its class if and only if it is not a degeneration of any other in its class.
It follows from the Hasse diagram in \cite{Se} that there are:
one rigid Lie algebra in $\mathcal{N}_{6,5}$,
three (two over $\C$) rigid Lie algebras in $\mathcal{N}_{6,4}$,
four (two over $\C$) in $\mathcal{N}_{6,3}$ and
three (two over $\C$) in $\mathcal{N}_{6,2}$.
The following table summarizes this information.

\

\begin{center}
\begin{tabular}{lllc}
       & $k$-step Lie algebra  & $k$-step Lie algebra  &  \\
$k$    & $\g$ as denoted in \cite{Se}  & $\g$ as denoted in \cite{CdGS}  & $\dim H_{k-nil}^2(\g, \g)$ \\
\hline \\[-2mm]
2      & $36$     & $\g_{6,26}$           & $0=18-18$ \\
       & $13+13$  & $\g_{6,22}$, $t=-1,1$ & $0=20-20$ \\[2mm]
3      & $246_E$  & $\g_{6,24}$, $t=-1,1$ & $2=26-24$ \\
       & $136_A$  & $\g_{6,19}$, $t=-1,1$ & $0=25-25$ \\[2mm]
4      & $1246$   & $\g_{6,13}$,          & $1=27-26$ \\
       & $1346_C$ & $\g_{6,21}$, $t=-1,1$ & $0=26-26$ \\[2mm]
5      & $12346_E$  & $\g_{6,14}$           & $0=28-28$ \\[2mm]
\end{tabular}

The list of rigid 6-dimensional nilpotent Lie algebras in $\mathcal{N}_{6,k}$
\end{center}

\

In this case there are three rigid nilpotent Lie algebras with non-zero cohomology $H_{k-nil}^2(\g, \g)$,
these are $\g_{6,13}$ and $\g_{6,24}$, $t=-1,1$;
and the non-zero cohomology classes correspond to infinitesimal solvable deformations.

\subsection{The two rigid curves in $\mathcal{N}_{7}$}\label{sec.curves in dim7}
It is known that there are no rigid Lie algebras in $\mathcal{N}_{7}$ (see \cite{Ca})
and that there are only three (two over $\C$) rigid curves of non-isomorphic Lie algebras
in $\mathcal{N}_{7}$ (see \cite{GA}).

One of these curves consists of 6-step nilpotent Lie algebras and it is denoted as
$\g_{I}(\alpha)$ by Burde \cite{Bu}, as $\g_{7,1.1}(ii_\lambda)$ by Magnin \cite{Ma}
and as $123457_I$ by Seeley \cite{Se2}. If $\g_6(r,t)$ is the surface of (solvable) Lie algebras
given by the following structure table
\begin{multline*}
\g_6(r, t):\;  ab = c,\;  ac = d,\;  ad = e,\;  ae = f,\;  af = g,\;  ag = rg, \\
bc = e,\;  bd = f,\;  be = rtf+(1-t)g,\;  bf = rg,\;  bg = r^2g,  \\ cd = -rtf+tg,
\end{multline*}
then   $\g_6(0,\alpha)$ is exactly $\g_{I}(\alpha)$.

The other two curves consist of 5-step nilpotent Lie algebras and they coincide over $\C$.
Over the complex numbers, this curve is denoted
as $\g_{1}(\lambda)$ by Burde \cite{Bu},
as $\g_{7,0.4}(\lambda)$ by Magnin \cite{Ma} and
as $12457_N$ by Seeley \cite{Se2}.
The structure table of $\g_{1}(\lambda)$ is obtained by setting $(r,t)=(0,\lambda)$ in
the following surface of solvable Lie algebras,
\begin{multline*}
\g_5(r,t):\;  ab=(1+tr)c,\;  ac=d,\;  ad=f+tg,\;  ae=g,\;  af=-rf+g,\; \\
            bc=e,\;    bd=g,\;  be=rd+f,\;   ce=g.
\end{multline*}

It is easy to check that both surfaces, $\{\g_5(r, t)\}_{(r,t)\in\R^2}$ and
$\{\g_6(r, t)\}_{(r,t)\in\R^2}$, are contained in $\mathcal{SN}_{7,5}$
(see \S\ref{sec.radical} for the definition). In addition, a Lie algebra in either
of these curves is nilpotent if and only if $r=0$.

Therefore if $\mu(r,t)$ is the Lie algebra structure of either $\g_5(r,t)$ or $\g_6(r,t)$
we consider the following 3-term $C^{\infty}$-chain complex
\begin{equation*}
\R^2\times GL(\g)\xrightarrow{\;\;F\;\;}\Lambda^2\g^* \otimes\g \xrightarrow{\;\;G=J\oplus SN_5\;\;}\Big(\Lambda^3\g^*\otimes\g\Big)
\oplus\Big((\g^*)^{\otimes 6}\otimes\g\Big),
\end{equation*}
where $SN_5$ is as in \eqref{SN_k} and $F(r,t,g)=g\cdot\mu(r,t)$.

The corresponding linear chain complex of the tangent spaces at the points $(r_0,t_0,I)\in \R^2\times GL(\g)$ and $\mu(r_0,t_0)\in \Lambda^2\g^* \otimes\g$ is
\begin{equation}\label{eq.exact_curves}
\R^2\times \g^* \otimes\g\xrightarrow{dF|_{(r_0,t_0,I)}}\Lambda^2\g^* \otimes\g \xrightarrow{dG|_{\mu(r_0,t_0)} }\Big(\Lambda^3\g^*\otimes\g\Big)
\oplus\Big((\g^*)^{\otimes 6}\otimes\g\Big)
\end{equation}
where
\begin{align*}
dF|_{(r_0,t_0,I)} &= [\partial_r|_{(r_0,t_0)}\mu,\;\partial_t|_{(r_0,t_0)}\mu,\;d_{\mu(r_0,t_0)}^1], \\
dG|_{\mu(r_0,t_0)} &= \left[\begin{matrix}d_{\mu(r_0,t_0)}^2 \\[2mm] d(SN_5)|_{\mu(r_0,t_0)}\end{matrix}\right].
\end{align*}
A computer calculation shows that the chain complex \eqref{eq.exact_curves} is exact
for all $(r_0,t_0)$ with $r_0t_0\ne -1$, if $\mu(r_0,t_0)$ is the structure of $\g_5(r_0,t_0)$;
and for all  $(r_0,t_0)$ with $t_0\ne 0$, if $\mu(r_0,t_0)$ is the structure of $\g_6(r_0,t_0)$.

\vspace{0.1cm}

The sizes of  $dF|_{(r_0,t_0,I)}$ and $dG|_{\mu(r_0,t_0)}$ are, respectively, $147\times 51$
and $7^7\times 147$. Writing down a computer code to obtain these matrices and confirm the
above claims is not such a difficult task. Doing this by hand would be a huge effort,
yet not impossible. For this, it could be convenient to use explicit ordinary cohomology
classes in $H^2(\mu(r_0,t_0),\mu(r_0,t_0))$ which can be found in \cite{Ma} (if $\mu(r_0,t_0)$ is
generic in either curve, then $\dim H^2(\mu(r_0,t_0),\mu(r_0,t_0))=9$). Then one should show that $d(SN_5)|_{\mu(r_0,t_0)}$ has no non-trivial kernel
within the space generated by these clases and the two tangent vectors corresponding to $\R^2$.

\vspace{0.1cm}

Next we use the exactness of \eqref{eq.exact_curves} to obtain the following proposition.

\begin{proposition}\label{prop.rigid_curve}
 The curves $\{\g_{1}(\lambda):\lambda\in\R\}$ and $\{\g_{I}(\alpha):\alpha\in\R,\;\alpha\ne 0\}$  are rigid curves
 in $\mathcal{N}_{7}$.
 Moreover, for any $r_0\ne0$, the curves
 $\{\g_5(r_0,t):t\in\R,\;t\ne -1/r_0\}$ and $\{\g_6(r_0,t):t\in\R,\;t\ne 0\}$ are rigid curves
 in $\mathcal{SN}_{7,5}$.
\end{proposition}

\begin{proof} Recall that $\g_6(0,t)\simeq\g_{I}(t)$ and $\g_5(0,t)\simeq\g_{1}(t)$. Fix $t_0\in\R$
(and $t_0\ne0$ if $\mu(0,t_0)$ is the structure of $\g_6(0,t_0)$).
It follows from Theorem \ref{mainthm} and the exactness of \eqref{eq.exact_curves} that there is a
neighborhood $U\subset\mathcal{SN}_{7,5}$ of $\mu(0,t_0)$ such that for any Lie algebra structure $\nu\in U$
there exists $(r,t)\in\R^2$ such that $\nu\simeq\mu(r,t)$.
If in addition $\nu$ is nilpotent (that is $\nu\in U\cap\mathcal{N}_{7}$), then $r$ must be 0 as $\mu(r,t)$ is
nilpotent if and only if $r=0$. This proves that $\{\g_{I}(\alpha):\alpha\in\R,\;\alpha\ne 0\}$ and $\{\g_{1}(\lambda):\lambda\in\R\}$ are rigid curves in $\mathcal{N}_{7}$.

\smallskip

Now fix $r_0\ne0$, and $t_0\ne -1/r_0$ if $\mu(r_0,t_0)$ is the structure of $\g_5(r_0,t_0)$;
or $t_0\ne0$ if $\mu(r_0,t_0)$ is the structure of $\g_6(r_0,t_0)$.
   It tuns out that a computer calculation shows that
   \eqref{eq.exact_curves} is still exact if we consider the function $F$ (and its differential)
 with the variable $r$ fixed at $r=r_0$.
Now Theorem \ref{mainthm} implies that there is a
neighborhood $U\subset\mathcal{SN}_{7,5}$ of $\mu(r_0,t_0)$ such that for any Lie algebra structure $\nu\in U$
there exists $t\in\R$ such that $\nu\simeq\mu(r_0,t)$.
This proves that  $\{\g_5(r_0,t):t\in\R,\;t\ne -1/r_0\}$ and $\{\g_6(r_0,t):t\in\R,\;t\ne 0\}$  are rigid curves
 in $\mathcal{SN}_{7,5}$.
\end{proof}

\begin{remark}
We point out that even when $r_1\ne r_2$, the orbit of the curve  $\{\mu(r_1,t):t\in\R\}$
might have non-empty intersection with the orbit of  $\{\mu(r_2,t):t\in\R\}$.
Therefore, even though the exactness of \eqref{eq.exact_curves} implies that the set $\{\mu(r,t):r,t\in\R\}$
is rigid in $\mathcal{SN}_{7,5}$,
this set does not constitute a rigid
 surface of \emph{pairwise non-isomorphic} Lie algebras in $\mathcal{SN}_{7,5}$.
\end{remark}

\subsection{Deformations and rigidity in $\mathcal{N}_{7,3}$}\label{sec.Nil_7-3}
The goal of this subsection is to obtain all rigid points and curves in $\mathcal{N}_{7,3}$.
According to the classification, there are more than fifty isomorphism classes and two 1-parameter curves. Specifically, in this subsection we do the following:
\begin{enumerate}[(a)]
\item With the assistance of a computer, we obtain all Lie algebras
$\g\in \mathcal{N}_{7,3}$ such that $H_{3-nil}^2(\g, \g)=0$ (there are four such
Lie algebras). According to Theorem \ref{k-rigidity} these Lie algebras are rigid in $\mathcal{N}_{7,3}$.

\smallskip

\item Similarly,  we obtain that $\dim H_{3-nil}^2(\g, \g)=1$ for $\g$ a generic point
in either one of the 1-parameter curves. The same argument given in the proof of Proposition \ref{prop.rigid_curve} shows that these curves are rigid in $\mathcal{N}_{7,3}$.

 \smallskip

 \item We also determine all the Lie algebras $\g\in \mathcal{N}_{7,3}$ that are not in
 the 1-parameter curves and have $\dim H_{3-nil}^2(\g, \g)=1$. For these Lie algebras we explicitly show a non-trivial deformation in $\mathcal{N}_{7,3}$
 (showing that they are not rigid).

 \smallskip

 \item As a byproduct, we point out a possible error in in Proposition 3.7 of \cite{GR}
 (see item (3) below).
\end{enumerate}
As a consequence of this, it is very likely that the four Lie algebras obtained in (a)
are exactly the rigid points in $\mathcal{N}_{7,3}$. In order to complete the proof of this,
one should consider the rigidity of all Lie algebras $\g\in \mathcal{N}_{7,3}$ such that
$\dim H_{3-nil}^2(\g, \g)\ge 2$ (approximately 40). We think that all of them have
non-trivial deformations in $\mathcal{N}_{7,3}$.

\smallskip

We will follow the classification of the 7-dimensional nilpotent Lie
algebras over $\R$ given by Gong in \cite{Go} and the one given by Seeley in \cite{Se2}.
The classification of Gong corrects some errors in the list given by Seeley. A more recent
classification is given by Magnin in \cite{Ma} (see also \cite{Ca}) but we will follow the classification of \cite{Go} and \cite{Se2} since these authors list the Lie algebras by
their upper central series (in \cite{Ma} the Lie algebras are listed by rank).
The list of 3-step nilpotent Lie algebras of dimension 7 in \cite{Go} has
52 isolated real Lie algebras and two 1-parameter families of pairwise non-isomorphic
nilpotent Lie algebras.

\medskip

\noindent
\emph{Rigid points and curves in $\mathcal{N}_{7,3}$.}
There are four  3-step nilpotent Lie algebras $\g$ with $H_{3-nil}^2(\g, \g)=0$ and thus
they are rigid in $\mathcal{N}_{7,3}$.
They are
\[
\g_{137B},\quad \g_{137B_1},\quad  \g_{247H},\quad \g_{247H_1}
\]
($\g_{137B}\simeq\g_{137B_1}$ and $\g_{247H}\simeq \g_{247H_1}$ over $\C$ \cite{Go}) and
the dimension of their orbits are, respectively, 36, 36, 38, 38.

\smallskip

The two 1-parameter families  in $\mathcal{N}_{7,3}$ are
\[
 \g_{147E}(t)\quad\text{ and }\quad \g_{147E_1}(t),\quad\text{ with $t>1$}
\]
(over $\C$, if $t=\cosh(\theta)>1$ then $\g_{147E_1}(t)$ is isomorphic to $\g_{147E}(t')$ with
$t'= -\frac{(1-i\sinh(\theta))^2}{\cosh^2(\theta)}\in\C$).
It turns out that, if
\[
 \g(t) \text{ is either }
 \begin{cases}
  \g_{147E_1}(t) & \text{with $t>1$, or} \\[2mm]
  \g_{147E}(t) & \text{with $t>1$, $t\ne2$,}
 \end{cases}
\]
then $\dim H_{3-nil}^2\big(\g(t), \g(t)\big)=1$
(and $\dim H_{3-nil}^2\big(\g_{147E}(2), \g_{147E}(2)\big)=3$),
and the non-zero cohomology class corresponds to the tangent vector
of $\g(t)$. Therefore, the same argument given in the proof of
Proposition \ref{prop.rigid_curve}, proves that
\[
\{\g_{147E}(t):1<t<2\}\qquad
\{\g_{147E}(t):2<t\}\qquad
\{\g_{147E_1}(t):1<t\}
\]
are rigid  curves in $\mathcal{N}_{7,3}$.

\medskip

\noindent
\emph{Non-trivial deformations of all $\g\in\mathcal{N}_{7,3}$ with $\dim H_{3-nil}^2\big(\g, \g\big)=1$.}
There are six other 3-step nilpotent Lie algebras $\g$ of dimension 7 in \cite{Go} (not members of the previous curves) such that $\dim H_{3-nil}^2\big(\g, \g\big)=1$. They are,
\[
 \g_{247G},\quad  \g_{247K},\quad \g_{147D},\quad \g_{137A},\quad \g_{137D},\quad \g_{137A_1}.
\]
None of them is rigid, in fact we claim that
\begin{align*}
 \g_{247H} & \rightarrow \g_{247K} & \text{ (see item \ref{it.Goze})}\\[-2mm]
           &  \;\begin{matrix}\searrow \\[3mm]  \end{matrix}\; \g_{247G}& \text{ (see item \ref{it.deg_GH})}\\
 \g_{137B} & \rightarrow \g_{137A}  & \text{ (see item \ref{it.137BD})} \\[-2mm]
           & \;\begin{matrix}\searrow \\[3mm]  \end{matrix}\;  \g_{137D}   & \text{ (see item \ref{it.137BD})} \\
 \g_{137B_1} & \rightarrow \g_{137A_1}  & \text{ (see item \ref{it.137BD})} \\
 \g_{147E_1}(t) & \rightarrow \g_{147D},\text{ as $t\to1$.} & \text{ (see item \ref{it.147E1D})}
\end{align*}

Next we prove these statements and point out an error that occur in \cite{GR}.

\vspace{0.1cm}

\begin{enumerate}[(1)]
\item \label{it.137BD}
It is not difficult to see that $\g_{137B}  \rightarrow \g_{137A}$ and $\g_{137B_1}  \rightarrow \g_{137A_1}$ since
\begin{align*}
\g_{137A}: & \; ab=e,\;  ae=g,\;  cd=f,\;  cf=g; \\
\g_{137B}: &  \; ab=e,\;  ae=g,\;  cd=f,\;  cf=g,\;  bd=g; \\
\g_{137A_1}: & \; ac=e,\;  ad=f,\;  ae=g,\;  bc=-f,\;  bd=e,\;  bf=g;\\
\g_{137B_1}: & \; ac=e,\;  ad=f,\;  ae=g,\;  bc=-f,\;  bd=e,\;  bf=g,\;  cd=g.
\end{align*}
In addition, let
\[
 \g(t): \; ab=e,\; ad=f,\; af=g,\; bc=f,\; bd=g,\; cd =-t^2e,\; ce=-g.
\]
If we rewrite the structure table of $\g(t)$ in the basis
\[
\{ ta+c,\; 2t(tb-d),\; -ta+c,\; -2t(tb+d),\; 4t^2(te-f),\; 4t^2(te+f),\; -8t^3g \}
\]
we obtain the structure table of $\g_{137B}$.
Since $\g(0)\simeq \g_{137D}$ (same structure table), it follows that $\g_{137B}  \rightarrow \g_{137D}$.

\medskip

\item \label{it.147E1D} The structure table of $\g_{147E_1}(t)$ is
\begin{multline*}
 \g_{147E_1}(t): \; ab=d,\;ac=-f,\;af=-tg, \\ bc=e,\;be=tg,\;bf=2g,\;cd=-2g.
\end{multline*}
If we set $t=1$ and rewrite this table in the basis
\[
\{ -a,\; a+b,\; c,\; -d,\;  e-f,\;  -f,\;  -g \}
\]
we obtain the structure table of  $\g_{147D}$
\[
 \g_{147D}: \; ab=d, \;ac=-f, \; ae= g, \; af=g, \; bc=e, \; bf=g, \;  cd=-2g.
\]

\medskip

\item \label{it.deg_GH} The structure table of  $\g_{247G}$ is that of $\g(0)$ where,
\begin{align*}
 \g(t):\;& ab = d,\; ac = e,\; ad = (1+\tfrac{t^3}{2})f+\tfrac{t^3}{2}g,\; ae = (1-\tfrac{t^3}{2})f-\tfrac{t^3}{2}g,\\
         & bd = f,\; be = g,\; cd = g,\; ce = f.
\end{align*}
On the other hand, if we rewrite the structure table of
\[
 \g_{247H}:\;ab=d,\; ac=e,\; ad=f,\; bd=f,\; be=g,\; cd=g,\; ce=f
\]
in the basis
\begin{multline*}
 \{
 2t^2a+(\tfrac12-t^2)b+\tfrac12 c,\;  \tfrac12(1+t)b+\tfrac12(1-t)c,\;  \tfrac12(1-t)b+\tfrac12(1+t)c, \\
 t^2(1+t)d+t^2(1-t)e,\; t^2(1-t)d+t^2(1+t)e, \\
 t^2(1+t^2)f+t^2(1-t^2)g,\;t^2(1-t^2)f+t^2(1+t^2)g
 \},
\end{multline*}
it coincides with the table of $\g(t)$. This shows that  $\g(t)\simeq \g_{247H}$ for $t\ne0$.

\medskip

 \item \label{it.Goze} In \cite{GR} it is claimed that $\g_{247K}$ is rigid
 in $\mathcal{N}_{7,3}$. Next, we will show that this is not the case.
 The Lie algebra claimed to be rigid in \cite{GR} is the following,
 \[
  \g: \; ab=c,\; ac=d,\; ae=f,\; af=g,\; bc=d,\; be=f,\; ce=g,\; ef=d.
 \]
Now, via the change of basis given by $\{b,-a+b,e,c,f,d,-g\}$ the structure table of $\g$
becomes the same as that of $g_{247K}$,
\[
 \g_{247K}: \; a b = d,\;  a c = e,\;  a d = f,\;  b e = g,\;  c d = g,\;  c e = f.
\]
Moreover, $\g_{247K}$ is $\g(0)$ for the curve of Lie algebras,
\[
 \g(t):\;a b = d,\;  a c = e,\;  a d = f,\;  b c = t^2e,\;  b e = g,\;  c d = g,\;  c e = f.
\]
On the other hand, the structure table of $\g_{247H}$ (which is rigid) is
\[
 \g_{247H}: \; a b = d,\;  a c = e,\;  a d = f,\;  b e = g,\;  b d = f,\;  c d = g,\;  c e = f
\]
and with respect to the basis given by
\[
\{-ia,-it^2(a-b),\;tc,\;t^2d,\; -ite,\; -it^2f,\;t^3g\},
\]
yields the structure table of $\g(t)$, showing that $\g(t)\simeq \g_{247H}$ over $\C$ for all $t\ne0$. Since $\g_{247H}\not\simeq \g_{247K}$ we obtain that the Lie algebra
$\g_{247K}$ is not rigid. This also shows that $\g_{247H} \rightarrow \g_{247K}$ as we
wanted to prove.

\end{enumerate}

\newcommand\bibit[5]{\bibitem[#1]
{#2}#3, {\em #4,\!\! } #5}

\end{document}